\documentclass[preprint,12pt]{elsarticle}

\usepackage[utf8]{inputenc}
\usepackage{lmodern}
\usepackage{amsmath,amssymb,amsthm,mathtools}
\usepackage{microtype}
\usepackage{enumitem}
\usepackage[colorlinks=true,linkcolor=blue,citecolor=blue,urlcolor=blue]{hyperref}
\usepackage[nameinlink,noabbrev]{cleveref}

\allowdisplaybreaks

\newtheorem{theorem}{Theorem}[section]
\newtheorem{proposition}[theorem]{Proposition}
\newtheorem{lemma}[theorem]{Lemma}
\newtheorem{corollary}[theorem]{Corollary}

\newtheorem{definition}[theorem]{Definition}

\newcommand{\R}{\mathbb{R}}
\newcommand{\N}{\mathbb{N}}
\newcommand{\E}{\mathbb{E}}
\newcommand{\Prob}{\mathbb{P}}
\newcommand{\cF}{\mathcal{F}}

\newcommand{\Id}{\mathrm{Id}}
\newcommand{\Per}{\operatorname{Per}}
\newcommand{\tr}{\operatorname{tr}}

\newcommand{\1}{\mathbf{1}}
\newcommand{\logp}{\log_{+}}

\newcommand{\abs}[1]{\left\lvert #1\right\rvert}

\newcommand{\ceil}[1]{\left\lceil #1\right\rceil}
\newcommand{\floor}[1]{\left\lfloor #1\right\rfloor}

\begin{document}

\begin{frontmatter}

\title{An Exponential Lower Bound for the Permanent\\of Random Bernoulli Matrix}

\author[aff1]{Yiming Chen\corref{cor1}}
\ead{ymchenmath@math.pku.edu.cn}
\cortext[cor1]{Corresponding author}
\affiliation[aff1]{organization={School of Mathematical Sciences, Peking University},
                  city={Beijing},
                  country={China}}

\begin{abstract}

Let $M_n$ be an $n\times n$ matrix with independent uniform sign entries. We prove that there exist absolute constants $C,c>0$ such that, for all sufficiently large $n$,
$$
\Prob\!\left(\abs{\Per(M_n)}\ge e^{-Cn}\sqrt{n!}\right)\ge 1-n^{-c}.
$$

This establishes the exponential scale lower bound suggested by Tao and Vu~\cite{TaoVu2009}. The proof bounds the cumulative logarithmic loss of the sum of squared permanents of minors under row exposure.


\end{abstract}


\end{frontmatter}

\section{Introduction and main results}

The asymptotic behavior of permanents of large random matrices has been
studied for several decades; see, for example,
\cite{BorovskikhKorolyuk1994,Rempala1996,RempalaWesolowski1999}.
Classical work established limit theorems and approximation results for
a variety of random permanent models. 

For the permanent of centered Bernoulli random matrix, let $M_n=(\xi_{ij})_{1\leq i,j\leq n},$ where the entries $\xi_{ij}$ are independent Rademacher random variables, i.e.,

\[
 \Prob(\xi_{ij}=1)=\Prob(\xi_{ij}=-1)=\frac12.
\]

The permanent of $M_n$ is defined by
\[
\operatorname{Per}(M_n)
=
\sum_{\sigma\in \mathbb S_n}
\prod_{i=1}^n \xi_{i,\sigma(i)}.
\]

A major advance was obtained by Tao and Vu~\cite{TaoVu2009}, who
determined the leading asymptotic order of the permanent.

\begin{theorem}[\cite{TaoVu2009}]\label{thm:permanent}
There is a positive constant \( c \) such that, for every \( \varepsilon > 0 \) and all sufficiently large \( n \) depending on \( \varepsilon \), then  
\[
\mathbb{P}\left(|\operatorname{Per}(M_n)| \geq n^{(\frac{1}{2} - \varepsilon)n}\right) \geq 1 - n^{-c}.
\]
\end{theorem}

They further pointed out that the bound $n^{(\frac{1}{2}-\varepsilon)n}$ could be sharpened to $n^{n/2}\exp(-\Theta(n))$, see Remark~1.4 of \cite{TaoVu2009}. Subsequent work on discrete random matrices has addressed several related aspects of the permanent. Kwan and Sauermann \cite{Kwan2022} extended the \(n^{n/2+o(n)}\) scale to random symmetric Bernoulli matrices, Ingram and Razborov \cite{IngramRazborov2026} studied the range of the permanent, and Hunter, Kwan, and Sauermann \cite{HunterKwanSauermann2025} established an exponential single point anti-concentration bound  
\[
    \sup_{z\in\mathbb R}
    \mathbb P\!\left(\operatorname{Per}(M_n)=z\right)
    \le e^{-cn}.
\]

The exponential scale lower bound suggested by Tao and Vu for the
original sign model, however, remained open. This improvement is more delicate than determining the $n^{n/2+o(n)}$ scale because of the fact that an error of \(o(n\log n)\) in \(\log|\operatorname{Per}(M_n)|\) suffices for the leading-order estimate, whereas the present result requires an \(O(n)\) error relative to \(\tfrac12\log(n!)\). The key difficulty is
therefore to handle the cumulative logarithmic loss over a linear
number of row exposure steps. Our main result confirms this scale.

\begin{theorem}\label{thm:main}
There exist absolute constants $C,c>0$ and $n_0\in\N$ such that, for every $n\geq n_0$,
\begin{equation}\label{eq:main-lower}
 \Prob\!\left(
   \abs{\Per(M_n)}\geq e^{-Cn}\sqrt{n!}
 \right)
 \geq 1-n^{-c}.
\end{equation}
\end{theorem}


Then we can derive the following
asymptotic bound.

\begin{corollary}\label{cor:two-sided}
There are absolute constants $0<c_1<C_1<\infty$ such that
\[
 \Prob\!\left(
 n^{n/2}e^{-C_1n}
 \leq \abs{\Per(M_n)}
 \leq n^{n/2}e^{-c_1n}
 \right)=1-o(1).
\]
\end{corollary}

\begin{proof}
The lower bound follows from Theorem \ref{thm:main} and Stirling's formula.
For the upper bound, expand the square of the permanent:
\[
 \E\abs{\Per(M_n)}^2
 =\sum_{\sigma,\tau\in \mathbb S_n}
   \E\prod_{i=1}^n \xi_{i,\sigma(i)}\xi_{i,\tau(i)}.
\]
If $\sigma\neq\tau$, there is an $i$ such that $\sigma(i)\neq\tau(i)$, and the expectation vanishes because the two signs in row $i$ are independent and centered.
If $\sigma=\tau$, the expectation equals $1$.
Hence
\begin{equation}\label{eq:second-moment}
 \E\abs{\Per(M_n)}^2=n!.
\end{equation}
By Markov's inequality,
\[
 \Prob\!\left(
 \abs{\Per(M_n)}>e^{n/4}\sqrt{n!}
 \right)\leq e^{-n/2}.
\]
Stirling's formula gives
\[
 e^{n/4}\sqrt{n!}
 =n^{n/2}\exp\!\left(-\frac n4+O(\log n)\right),
\]
which is at most $n^{n/2}e^{-c_1n}$ for a suitable absolute $c_1>0$ and all sufficiently large $n$.
\end{proof}
\subsection*{Proof architecture}

Tao and Vu~\cite{TaoVu2009} propagate large minors by increasing either their number or a common lower bound on the absolute permanents. We retain their row-exposure framework, but in the main growth phase track the squared minor energy

\begin{equation}\label{eq:def-Sk}
 S_k:=\sum_{A\in\binom{[n]}{k}}X_k(A)^2,
\end{equation}
where $X_k(A)$ is the permanent of the first $k$ rows restricted
to columns $A$, with $X_0(\varnothing)=1$.

The key observation is that, conditional on the first $k$ rows,
$S_{k+1}=\varepsilon^{\mathsf T}Q_k\varepsilon$, where $\varepsilon$
is the next row and $Q_k\succeq0$ satisfies
\[
 \tr Q_k=(n-k)S_k,
 \quad \text{and} \quad 
 \max_i(Q_k)_{ii}\leq\frac{\tr Q_k}{n-k}.
\]
Thus the diagonal remains flat even when the energy is concentrated
in very few minors. Up to $k_*:=n-\ceil{n/\log n}$, our logarithmic
small-ball bound gives a uniform bound on the conditional
expected truncated logarithmic loss relative to $(n-k)S_k$.
A separate one step lower bound removes the truncation with high
probability. Martingale concentration then bounds the cumulative
logarithmic loss, showing that the resulting total loss is $O(n)$ with high probability.
Iterating from $S_0=1$ yields
$S_{k_*}\geq e^{-Cn}n!/(n-k_*)!$, and averaging produces
a minor of magnitude at least $e^{-Cn}\sqrt{k_*!}$.
The Tao--Vu endgame, with its parameters verified at this cutoff,
incurs only an additional $O((\log n)^2)$ logarithmic loss.
Since $\log(n!/k_*!)=O(n)$, we obtain the desired
$e^{-Cn}\sqrt{n!}$ lower bound.
%

\section{Probabilistic preliminaries}\label{sec:prelim}

We introduce some basic results in this section.
Unless stated otherwise, $C,c$ and their indexed variants denote positive absolute constants whose values may change from line to line.

\subsection{Littlewood--Offord anti-concentration}

We recall the classical theorem of Erd\H{o}s~\cite{Erdos1945} in the following interval form.

\begin{theorem}[Erd\H{o}s--Littlewood--Offord]\label{thm:ELO}
Let $a_1,\dots,a_N\in\R$, let $\varepsilon_1,\dots,\varepsilon_N$ be independent uniform signs, and suppose that $\abs{a_i}\geq \rho$ for at least $K$ indices $i$, where $\rho>0$.
Then, for every $x\in\R$,
\begin{equation}\label{eq:ELO}
 \Prob\!\left(
   \abs{\sum_{i=1}^N\varepsilon_i a_i-x}\leq \rho
 \right)
 \leq \frac{C}{\sqrt K}.
\end{equation}
\end{theorem}

After conditioning on some of the signs, \eqref{eq:ELO} remains valid provided that at least $K$
of the remaining signs have coefficients of absolute value at least
$\rho$, and that the conditioned terms are absorbed into the shift $x$.



\subsection{Hoeffding, bounded differences, and martingale concentration}

We use the following standard consequences of Hoeffding's inequality~\cite{Hoeffding1963}, the bounded differences inequality~\cite{McDiarmid1989}, and Azuma's martingale inequality~\cite{Azuma1967}.

\begin{theorem}\label{thm:hoeffding}
Let $Z_1,\dots,Z_r$ be independent $\{0,1\}$-valued random variables satisfying $\E Z_i\geq 1/2$.
Then
\begin{equation}\label{eq:hoeffding-third}
 \Prob\!\left(\sum_{i=1}^r Z_i<\frac r3\right)
 \leq e^{-r/18}.
\end{equation}
Also,
\begin{equation}\label{eq:none-success}
 \Prob(Z_1=\cdots=Z_r=0)\leq 2^{-r}.
\end{equation}
\end{theorem}

\begin{theorem}\label{thm:mcdiarmid}
Let $Y=f(\eta_1,\dots,\eta_N)$, where $\eta_1,\dots,\eta_N$ are independent.
Suppose that changing only the $i$th coordinate can change $f$ by at most $c_i$.
Then, for every $t\geq0$,
\begin{equation}\label{eq:mcdiarmid}
 \Prob(Y\leq \E Y-t)
 \leq \exp\!\left(-\frac{2t^2}{\sum_{i=1}^N c_i^2}\right).
\end{equation}
\end{theorem}

\begin{theorem}\label{thm:azuma}
Let $(\mathcal G_j)_{j=0}^N$ be a filtration and let $(Z_j)_{j=0}^N$ be a martingale such that $ \abs{Z_j-Z_{j-1}}\leq H$ almost surely for every $1\leq j\leq N$. Then, for every $t\geq0$,
\begin{equation}\label{eq:azuma}
 \Prob(Z_N-Z_0\geq t)
 \leq \exp\!\left(-\frac{t^2}{2NH^2}\right).
\end{equation}
\end{theorem}

We shall also use the following first moment selection principle.

\begin{lemma}\label{lem:first-moment}
Let $E_1,\dots,E_q$ be arbitrary events, not necessarily independent, satisfying $\Prob(E_i^c\mid\mathcal G)\leq\delta$ almost surely for every $i$, where $\mathcal G$ is a sigma algebra.
Then
\begin{equation}\label{eq:first-moment-selection}
 \Prob\!\left(
   \#\{i:E_i^c\}>\frac q2
   \,\middle|\,\mathcal G
 \right)
 \leq 2\delta.
\end{equation}
\end{lemma}

\begin{proof}
Note that $\E(\sum_{i=1}^q\1_{E_i^c}\mid\mathcal G)\leq q\delta$.
Conditional Markov gives
\[
 \Prob(\sum_{i=1}^q\1_{E_i^c}>q/2\mid\mathcal G)
 \leq \frac{2}{q}\E(\sum_{i=1}^{q}\mathbf 1_{E_i^c}\mid\mathcal G)
 \leq2\delta.
\]

\end{proof}

\subsection{A logarithmic estimate for Rademacher quadratic forms}
\label{sec:flat-diagonal}

For positive-semidefinite Rademacher quadratic forms, we bound the
expected truncated logarithmic loss in terms of the largest diagonal
entry relative to the trace. The proof uses the following scalar
Berry--Esseen bound.

\begin{theorem}
\label{thm:raic}
Let $a=(a_1,\ldots,a_N)\in\R^N\setminus\{0\}$, and let
$\varepsilon_1,\ldots,\varepsilon_N$ be independent uniform signs.
There is an absolute constant $C>0$ such that, for every $t\geq0$,
\begin{equation}\label{eq:raic}
\Prob\left(\left|\sum_{i=1}^N a_i\varepsilon_i\right|\leq t\right)
\leq C\left(\frac{t}{\|a\|_2}
+\frac{\sum_{i=1}^N|a_i|^3}{\|a\|_2^3}\right).
\end{equation}
\end{theorem}

\begin{proof}
Apply \cite[Theorem~1.1]{Raic2019} in dimension one to
$a_i\varepsilon_i/\|a\|_2$ and the interval
$[-t/\|a\|_2,t/\|a\|_2]$. The approximation error is at most
$C\|a\|_2^{-3}\sum_{i=1}^N|a_i|^3$, and the Gaussian probability of
this interval is at most $Ct/\|a\|_2$.
\end{proof}

For $x>0$, write $\logp x:=\max\{\log x,0\}$.

\begin{theorem}
\label{thm:flat-diagonal}
Let $Q=(q_{ij})$ be a real symmetric positive-semidefinite $N\times N$
matrix with $\tr Q>0$ and
$\max_{1\leq i\leq N}q_{ii}\leq\tr Q/r$ for some $r\geq1$.
Let $\varepsilon\in\{-1,1\}^N$ be a uniform sign vector.
There is an absolute constant $C>0$ such that, for every truncation
level $H\geq1$,
\begin{equation}\label{eq:flat-diagonal-conclusion}
\E\logp\left(
\frac{\tr Q}{
\max\{\varepsilon^{\mathsf T}Q\varepsilon,e^{-H}\tr Q\}}
\right)
\leq C+CHr^{-1/8}.
\end{equation}
\end{theorem}

\begin{proof}
By homogeneity, we may assume that $\tr Q=1$.
Let $G$ be a centered Gaussian vector with covariance $Q$, independent of
$\varepsilon$, and fix $0<u,\delta\leq1$.
Conditionally on $\varepsilon$, the scalar $\langle G,\varepsilon\rangle$
is Gaussian with variance $\varepsilon^{\mathsf T} Q\varepsilon$, so
\begin{equation}\label{eq:gaussian-randomization}
\Prob(\varepsilon^{\mathsf T} Q\varepsilon\leq u)
\leq C\,\Prob\bigl(|\langle G,\varepsilon\rangle|\leq\sqrt u\bigr).
\end{equation}
Conditioning instead on $G$, Theorem~\ref{thm:raic} gives,
on $\{\|G\|_2^2\geq\delta\}$,
\begin{equation}\label{eq:conditional-scalar-small-ball}
\Prob\bigl(|\langle G,\varepsilon\rangle|\leq\sqrt u\mid G\bigr)
\leq C\sqrt{\frac{u}{\delta}}
+C\delta^{-3/2}\sum_{i=1}^N|G_i|^3.
\end{equation}
Since $G_i\sim N(0,q_{ii})$ and $\max_i q_{ii}\leq r^{-1}$,
\begin{equation}\label{eq:random-coefficient-third-moment}
\E\sum_{i=1}^N|G_i|^3
\leq C\sum_{i=1}^Nq_{ii}^{3/2}
\leq Cr^{-1/2}\sum_{i=1}^Nq_{ii}
=Cr^{-1/2}.
\end{equation}
Also, Markov's inequality and the Gaussian Laplace transform yield
\begin{equation}\label{eq:gaussian-variance-small-ball}
\begin{aligned}
\Prob(\|G\|_2^2<\delta)
&\leq e\,\E e^{-\|G\|_2^2/\delta}\\
&=e\,\det\!\left(\Id_N+\frac{2Q}{\delta}\right)^{-1/2}
\leq C\sqrt\delta.
\end{aligned}
\end{equation}
Here $\det(\Id_N+2Q/\delta)\geq1+2/\delta$, since the eigenvalues
of $Q$ are nonnegative and sum to one. Splitting according to $\{\|G\|_2^2\geq\delta\}$ and combining
\eqref{eq:gaussian-randomization}--\eqref{eq:gaussian-variance-small-ball},
with $\delta=\max\{u^{1/2},r^{-1/4}\}$, we obtain

\begin{align}
\Prob(\varepsilon^{\mathsf T} Q\varepsilon\leq u)
&\leq C\left(\sqrt\delta+\sqrt{\frac{u}{\delta}}
+r^{-1/2}\delta^{-3/2}\right)\label{eq:flat-diagonal-intermediate}\\
&\leq C\bigl(u^{1/4}+r^{-1/8}\bigr).\label{eq:flat-diagonal-small-ball}
\end{align}
Integrating \eqref{eq:flat-diagonal-small-ball} with $u=e^{-t}$ gives
\begin{equation}\label{eq:log-tail-integral-Q}
\begin{aligned}
\E\logp\left(
\frac{1}{\max\{\varepsilon^{\mathsf T}Q\varepsilon,e^{-H}\}}
\right)
&=\int_0^H
  \Prob(\varepsilon^{\mathsf T}Q\varepsilon\leq e^{-t})\,dt\\
&\leq C\int_0^H e^{-t/4}\,dt+CHr^{-1/8}\\
&\leq C+CHr^{-1/8}.
\end{aligned}
\end{equation}
\end{proof}

\section{Logarithmic propagation of the squared minor energy}
\label{sec:energy-propagation}

In this section, we bound the cumulative logarithmic loss of the squared-minor energy
relative to its conditional mean growth. Let $\cF_k$ be the sigma-algebra generated by the first $k$ rows, with
$\cF_0$ trivial. Recall that
$S_k=\sum_{A\in\binom{[n]}{k}}X_k(A)^2$ and
$X_0(\varnothing)=S_0=1$. Set $k_*:=n-r_*$ with $r_*=\ceil{\frac{n}{\log n}}$.


\begin{proposition}\label{prop:large-near-top-minor}
There are absolute constants $C,c>0$ such that, for every sufficiently
large $n$, with probability at least $1-e^{-cn/(\log n)^2}$, there exists
$A_*\in\binom{[n]}{k_*}$ satisfying
\begin{equation}\label{eq:large-kstar-minor}
\abs{X_{k_*}(A_*)}\geq e^{-Cn}\sqrt{k_*!}.
\end{equation}
The set $A_*$ may be chosen $\cF_{k_*}$-measurably.
\end{proposition}

\subsection{Quadratic form representation and diagonal flatness}

For $0\leq k<n$, define the symmetric $n\times n$ matrix $Q_k$ by
\begin{equation}\label{eq:Lx}
z^{\mathsf T}Q_kz
=\sum_{B\in\binom{[n]}{k+1}}
  \left(\sum_{i\in B}z_iX_k(B\setminus\{i\})\right)^2,
\qquad z\in\R^n.
\end{equation}
It is $\cF_k$-measurable and positive semidefinite.
Conditionally on $\cF_k$, the next row
$\varepsilon=(\xi_{k+1,1},\ldots,\xi_{k+1,n})$ is a uniform sign vector.
Expanding each minor along this row gives
\begin{equation}\label{eq:Sk-quadratic}
S_{k+1}=\varepsilon^{\mathsf T}Q_k\varepsilon.
\end{equation}

\begin{lemma}\label{lem:trace-diagonal}
For every $0\leq k<n$, $\tr Q_k=(n-k)S_k$, and
\begin{align}
(Q_k)_{ii}
&=\sum_{\substack{A\in\binom{[n]}{k}\\i\notin A}}X_k(A)^2
\leq S_k,\qquad i\in[n].\label{eq:diag-Qk}
\end{align}
In particular, on $\{S_k>0\}$,
\begin{equation}\label{eq:flat-Qk}
\max_i(Q_k)_{ii}\leq\frac{\tr Q_k}{n-k}.
\end{equation}
\end{lemma}

\begin{proof}
Comparing the coefficients of $z_i^2$ in \eqref{eq:Lx}
gives \eqref{eq:diag-Qk}. Summing over $i$ counts each
$X_k(A)^2$ exactly $n-k$ times, so $\tr Q_k=(n-k)S_k$.
Together with $(Q_k)_{ii}\leq S_k$, this gives
\eqref{eq:flat-Qk}.

\end{proof}

Since the signs are independent and centered,
$\E(S_{k+1}\mid\cF_k)=(n-k)S_k$.

\subsection{A one-step lower bound}

\begin{lemma}\label{lem:one-step-floor}
For every $0\leq k<n$, on $\{S_k>0\}$,
\begin{equation}\label{eq:one-step-floor-prob}
\Prob\left(
S_{k+1}<\frac{n-k}{6(k+1)}S_k\,\middle|\,\cF_k
\right)\leq2e^{-(n-k)/18}.
\end{equation}
\end{lemma}

\begin{proof}
Condition on $\cF_k$ and suppose that $S_k>0$.
Fix $A\in\binom{[n]}{k}$. For $i\notin A$, expansion along the next row gives
\begin{equation}\label{eq:child-expansion-fixed-parent}
X_{k+1}(A\cup\{i\})
=\varepsilon_iX_k(A)
 +\sum_{j\in A}\varepsilon_j
      X_k((A\cup\{i\})\setminus\{j\}).
\end{equation}
After conditioning further on $(\varepsilon_j)_{j\in A}$, the second term in \eqref{eq:child-expansion-fixed-parent}
is fixed and the signs $(\varepsilon_i)_{i\notin A}$ remain independent.
Since $\max\{|c+x|,|c-x|\}\geq|x|$ for real $c,x$,
\begin{equation}\label{eq:child-success-half}
\Prob\left(
\abs{X_{k+1}(A\cup\{i\})}\geq\abs{X_k(A)}
\,\middle|\,\cF_k,(\varepsilon_j)_{j\in A}
\right)\geq\frac12.
\end{equation}
These events are conditionally independent. Let $G_A$ be the event that
at least $(n-k)/3$ of the minors $X_{k+1}(A\cup\{i\})$, $i\notin A$,
have absolute value at least $\abs{X_k(A)}$.
Theorem~\ref{thm:hoeffding}, followed by averaging over the additional
conditioning, gives 
\begin{equation}\label{eq:GA-fail}
\Prob(G_A^c\mid\cF_k)\leq e^{-(n-k)/18}.
\end{equation}
Conditional Markov's inequality now yields
\begin{equation}\label{eq:Bk-markov}
\begin{aligned}
\Prob\left(
\sum_{A\in\binom{[n]}{k}}X_k(A)^2\1_{G_A^c}
>\frac{S_k}{2}\,\middle|\,\cF_k
\right)
&\leq\frac{2}{S_k}
\sum_{A\in\binom{[n]}{k}}X_k(A)^2\Prob(G_A^c\mid\cF_k)\\
&\leq2e^{-(n-k)/18}.
\end{aligned}
\end{equation}
On $\{\sum_{A\in\binom{[n]}{k}}
X_k(A)^2\1_{G_A^c}\leq S_k/2\}$, double counting gives
\[
\begin{aligned}
(k+1)S_{k+1}
&=\sum_{A\in\binom{[n]}{k}}\sum_{i\notin A}
  X_{k+1}(A\cup\{i\})^2\\
&\geq\frac{n-k}{3}
  \sum_{A\in\binom{[n]}{k}}X_k(A)^2\1_{G_A}
\geq\frac{n-k}{6}S_k.
\end{aligned}
\]
\end{proof}

\subsection{Accumulated logarithmic loss}

For $0\leq k<k_*$, define
\begin{equation}\label{eq:actual-loss}
\ell_k=
\begin{cases}
\displaystyle\logp\frac{(n-k)S_k}{S_{k+1}},
    &S_k>0,\ S_{k+1}>0,\\[1ex]
+\infty,&S_k>0,\ S_{k+1}=0,\\
0,&S_k=0,
\end{cases}
\end{equation}
and truncate at height $\log(6n)$: $\widetilde\ell_k:=\min\{\ell_k,\log(6n)\}.$

\begin{proposition}\label{prop:conditional-log-loss}
There is an absolute constant $C_0>0$ such that, for every sufficiently
large $n$ and every $0\leq k<k_*$,
\begin{equation}\label{eq:conditional-log-loss}
\E(\widetilde\ell_k\mid\cF_k)\leq C_0
\qquad\text{almost surely}.
\end{equation}
\end{proposition}

\begin{proof}
On $\{S_k=0\}$ the loss is zero.
On $\{S_k>0\}$, Lemma~\ref{lem:trace-diagonal} and
\eqref{eq:Sk-quadratic} allow us to apply
Theorem~\ref{thm:flat-diagonal} conditionally on $\cF_k$, with
$r=n-k$ and $H=\log(6n)$. Thus
\[
\E(\widetilde\ell_k\mid\cF_k)
\leq C+C\log(6n)(n-k)^{-1/8}.
\]
Since $n-k\geq r_*\geq n/\log n$,
\begin{equation}\label{eq:BE-error-vanish}
\log(6n)(n-k)^{-1/8}
\leq C\frac{(\log n)^{9/8}}{n^{1/8}}=o(1)
\end{equation}
uniformly over $k<k_*$.
\end{proof}

\begin{proposition}\label{prop:linear-total-loss}
There are absolute constants $C,c>0$ such that, for every sufficiently
large $n$,
\begin{equation}\label{eq:linear-total-loss}
\Prob\left(\sum_{k=0}^{k_*-1}\ell_k\leq Cn\right)
\geq1-e^{-cn/(\log n)^2}.
\end{equation}
\end{proposition}

\begin{proof}
The partial sums of
$\widetilde\ell_k-\E(\widetilde\ell_k\mid\cF_k)$ form a martingale
with respect to $(\cF_j)_{j=0}^{k_*}$, whose increments are bounded in
absolute value by $\log(6n)$.
Theorem~\ref{thm:azuma} and Proposition~\ref{prop:conditional-log-loss},
with $k_*\leq n$ and deviation $n$, give
\begin{equation}\label{eq:azuma-log-loss}
\Prob\left(
\sum_{k=0}^{k_*-1}\widetilde\ell_k>(C_0+1)n
\right)
\leq\exp\left(-\frac{n}{2\log^2(6n)}\right).
\end{equation}
For $0\leq k<k_*$, since $k+1\leq n$,
\begin{equation}\label{eq:bad-floor-event}
\{\ell_k>\log(6n)\}
\subseteq\left\{S_k>0,\quad
S_{k+1}<\frac{n-k}{6(k+1)}S_k\right\}.
\end{equation}
Lemma~\ref{lem:one-step-floor} and a union bound therefore imply
\begin{equation}\label{eq:bad-floor-union}
\Prob\left(\max_{0\leq k<k_*}\ell_k>\log(6n)\right)
\leq2n e^{-r_*/18}.
\end{equation}
On $\{\max_{0\leq k<k_*}\ell_k\leq\log(6n)\}$,
\begin{equation}\label{eq:loss-equals-truncated}
\ell_k=\widetilde\ell_k\leq\log(6n),\qquad 0\leq k<k_*.
\end{equation}
Combining \eqref{eq:azuma-log-loss} and
\eqref{eq:bad-floor-union} gives
\[
\Prob\left(\sum_{k=0}^{k_*-1}\ell_k>(C_0+1)n\right)
\leq
\exp\left(-\frac{n}{2\log^2(6n)}\right)
+2ne^{-r_*/18}
\leq e^{-cn/(\log n)^2}
\]
for some absolute $c>0$ and all sufficiently large $n$.
\end{proof}

\begin{proof}[Proof of Proposition~\ref{prop:large-near-top-minor}]
On $\{\sum_{k=0}^{k_*-1}\ell_k\leq Cn\}$, whose probability
is bounded below in \eqref{eq:linear-total-loss},
all energies $S_0,\ldots,S_{k_*}$ are positive:
the first zero would make the preceding loss infinite.
For $0\leq k<k_*$,
\begin{equation}\label{eq:logR-lower}
\log\frac{S_{k+1}}{(n-k)S_k}\geq-\ell_k.
\end{equation}
Iterating from $S_0=1$ gives
\begin{equation}\label{eq:Skstar-lower}
S_{k_*}
\geq\exp\left(-\sum_{k=0}^{k_*-1}\ell_k\right)\frac{n!}{r_*!}
\geq e^{-Cn}\frac{n!}{r_*!}.
\end{equation}
Consequently,
\begin{equation}\label{eq:max-kstar-square}
\max_{A\in\binom{[n]}{k_*}}X_{k_*}(A)^2
\geq\frac{S_{k_*}}{\binom{n}{k_*}}
\geq e^{-Cn}k_*!.
\end{equation}
Choose $A_*$ as the lexicographically first maximizer of
$\abs{X_{k_*}(A)}$. This choice is $\cF_{k_*}$-measurable, and taking
square roots proves the proposition after adjusting $C$.
\end{proof}

\section{The endgame}
\label{sec:endgame}
\label{app:endgame}

Starting from a minor of order $k_*$, we obtain a
lower bound for the full permanent with a further logarithmic loss of
$O((\log n)^2)$ and conditional failure probability at most $n^{-c}$. 
We follow Tao and Vu~\cite[Section~6]{TaoVu2009},
checking the estimates at $k_*$
rather than at a fixed fraction of $n$.

\begin{proposition}\label{prop:endgame}
There is an absolute constant $c>0$ such that, for all sufficiently large
$n$, every $\cF_{k_*}$-measurable $A_*\in\binom{[n]}{k_*}$ and every
$\lambda>0$ satisfy, almost surely on
$\{\abs{X_{k_*}(A_*)}\geq\lambda\}$,
\begin{equation}\label{eq:endgame-conclusion}
\Prob\left(
\abs{\Per(M_n)}\geq n^{-L}\lambda
\,\middle|\,\cF_{k_*}
\right)\geq1-n^{-c},
\end{equation}
where $L=\floor{\frac{\log n}{100}}.$
\end{proposition}

\begin{definition}[Heavy minors]\label{def:heavy}
For $A\in\binom{[n]}{k}$ and $\lambda>0$, call $A$ $\lambda$-heavy if
$\abs{X_k(A)}\geq\lambda$. A set $A\cup\{i\}$ with $i\notin A$ is a
child of $A$, and $A$ is one of its parents.
\end{definition}

For integers $1\leq j\leq n$ and $N\geq0$, and for $\lambda>0$, let
$\mathcal E(j,N,\lambda)$ be the event that there are $N$ distinct
$\lambda$-heavy sets in $\binom{[n]}{n-j}$ with pairwise disjoint
complements in $[n]$. Choices from finite families are made
lexicographically, using only the rows exposed at that stage.

\subsection{Producing minors with disjoint complements}

\begin{lemma}[cf.\ {\cite[Lemma~4.2]{TaoVu2009}}]
\label{lem:heavy-parent-children}
Fix $0\leq k<n$ and condition on $\cF_k$. If
$A\in\binom{[n]}{k}$ is $\lambda$-heavy and
$I\subseteq[n]\setminus A$, then
\begin{equation}\label{eq:no-heavy-child}
\Prob\left(
\abs{X_{k+1}(A\cup\{i\})}<\lambda\text{ for all }i\in I
\,\middle|\,\cF_k
\right)\leq2^{-\abs I}.
\end{equation}
\end{lemma}

\begin{proof}
Condition further on the entries of row $k+1$ outside $I$.
By \eqref{eq:child-expansion-fixed-parent}, each child permanent depends
only on its own remaining sign, with coefficient $X_k(A)$. At least one
of the two sign choices gives absolute value at least $\lambda$.
Conditional independence gives the bound $2^{-\abs I}$, and averaging
over the fixed signs proves the claim.
\end{proof}

\begin{lemma}[cf.\ {\cite[Lemma~6.1]{TaoVu2009}}]
\label{lem:prescribed-block}
Let $L\geq1$ and $0\leq k\leq n-3L$ be integers. Condition on $\cF_k$
and suppose that $A_0\in\binom{[n]}{k}$ is $\lambda$-heavy.
For any $B\subseteq[n]\setminus A_0$ with $\abs B=2L$, with conditional
probability at least $1-Ce^{-cL}$ there is a $\lambda$-heavy set
$A\in\binom{[n]}{n-L}$ such that
\begin{equation}\label{eq:prescribed-block-complement}
[n]\setminus A\subseteq B.
\end{equation}
\end{lemma}

\begin{proof}
Starting with $A_k=A_0$, construct $A_{j+1}$ after revealing row $j+1$.
Choose a $\lambda$-heavy child by adding a column outside $B$ whenever possible;
otherwise choose any heavy child, or any child if none is heavy.
Write $W_j:=\abs{[n]\setminus(B\cup A_j)}$ and
$j_0:=n-\ceil{5L/2}\geq k$.

Until the first failure to choose a heavy child outside $B$, the sets
remain heavy and $W_j$ decreases by one at each step. By
Lemma~\ref{lem:heavy-parent-children}, the probability of such a failure
before reaching level $j_0$ is at most
$\sum_{w=\ceil{L/2}+1}^{\infty}2^{-w}=2^{-\ceil{L/2}}$. With conditional probability at least $1-2^{-\ceil{L/2}}$,
the set $A_{j_0}$ is $\lambda$-heavy and
\begin{equation}\label{eq:Wj0}
W_{j_0}=\ceil{L/2}.
\end{equation}

Fix a realization of $\cF_{j_0}$ for which
$\abs{X_{j_0}(A_{j_0})}\geq\lambda$ and
$W_{j_0}=\ceil{L/2}$. For $j_0\leq j\leq n-L$, let
$E_j$ be the event that all selected sets from level $j_0$ through level
$j$ are heavy.  On $E_j$, apply \eqref{eq:no-heavy-child}
with $I=[n]\setminus A_j$.
A union bound gives
\begin{equation}\label{eq:lose-heavy-union}
\Prob(E_{n-L}^c\mid\cF_{j_0})
\leq\sum_{j=j_0}^{n-L-1}2^{-(n-j)}\leq2^{-L}.
\end{equation}
By \eqref{eq:no-heavy-child} with
$I=[n]\setminus(B\cup A_j)$, on $E_j\cap\{W_j>0\}$,
a $\lambda$-heavy child obtained by adding a column outside $B$
exists with conditional probability at least $1-2^{-W_j}\geq1/2$. Choosing it decreases $W_j$ by
one. Since $W_j$ cannot increase and $E_{j+1}\subseteq E_j$, it follows
that
\[
\E\left((2^{W_{j+1}}-1)\1_{E_{j+1}}\mid\cF_j\right)
\leq\frac34(2^{W_j}-1)\1_{E_j}.
\]
Iterating over $n-L-j_0\geq3L/2$ steps yields
\begin{equation}\label{eq:W-not-zero}
\begin{aligned}
\Prob(E_{n-L}\cap\{W_{n-L}>0\}\mid\cF_{j_0})
&\leq\E\left((2^{W_{n-L}}-1)\1_{E_{n-L}}\mid\cF_{j_0}\right)\\
&\leq2^{L/2+1}(3/4)^{3L/2}
=2(27/32)^{L/2}.
\end{aligned}
\end{equation}
Combining the bound $2^{-\ceil{L/2}}$
with \eqref{eq:lose-heavy-union} and \eqref{eq:W-not-zero},
and averaging the latter two estimates over $\cF_{j_0}$
on first-phase success, bounds the conditional failure probability by
$2^{-\ceil{L/2}}+2^{-L}+2(27/32)^{L/2}\leq Ce^{-cL}$.
Thus, with conditional probability at least $1-Ce^{-cL}$,
the set $A_{n-L}$ is $\lambda$-heavy and
$[n]\setminus A_{n-L}\subseteq B$.

\end{proof}

\begin{proposition}[cf.\ {\cite[Corollary~6.2]{TaoVu2009}}]
\label{prop:many-codim-L}
Under the assumptions of
Lemma~\ref{lem:prescribed-block},
\[
\Prob\left(
\mathcal E\left(L,\floor{\frac{n-k}{4L}},\lambda\right)
\,\middle|\,\cF_k
\right)\geq1-Ce^{-cL}.
\]
\end{proposition}

\begin{proof}
Choose $q=\floor{(n-k)/(2L)}$ disjoint $2L$-element blocks in
$[n]\setminus A_0$. By Lemma~\ref{lem:prescribed-block}, the conditional
expected number of unsuccessful blocks is at most $Cqe^{-cL}$.
Markov's inequality shows that at least $q/2$ blocks succeed, except
with conditional probability $2Ce^{-cL}$. Choose one witnessing minor
from each successful block. Their complements are disjoint, and
$\floor{q/2}=\floor{(n-k)/(4L)}$.
\end{proof}

\subsection{Reducing the codimension by one}

\begin{lemma}[cf.\ {\cite[Lemma~6.3]{TaoVu2009}}]
\label{lem:codimension-reduction}
There are absolute constants $c_0,C_0>0$ such that, for all sufficiently
large $n$, all integers $2\leq j\leq n$ and $N\geq n^{1/2}$, and all
$\lambda>0$, almost surely on $\mathcal E(j,N,\lambda)$,
\begin{equation}\label{eq:codim-reduction-prob}
\Prob\left(
\mathcal E\left(j-1,\floor{N/10},\lambda/n\right)
\,\middle|\,\cF_{n-j}
\right)\geq1-C_0n^{-c_0}.
\end{equation}
\end{lemma}

\begin{proof}
Condition on $\cF_{n-j}$ and choose witnessing sets $A_1,\ldots,A_N$.
Choose $h_i\in A_i^c$ and set $B_i:=A_i\cup\{h_i\}$.
The $h_i$ are distinct, and the complements $B_i^c$ remain pairwise
disjoint and nonempty. Call $B_i$ good if
\begin{equation}\label{eq:T-endgame}
\#\left\{h\in B_i:
\abs{X_{n-j}(B_i\setminus\{h\})}\geq\lambda/n\right\}\geq n^{1/10},
\end{equation}
and bad otherwise.
Writing $\varepsilon$ for row $n-j+1$, cofactor expansion gives
\begin{equation}\label{eq:Bi-cofactor}
X_{n-j+1}(B_i)
=\sum_{h\in B_i}\varepsilon_hX_{n-j}(B_i\setminus\{h\}).
\end{equation}

Suppose first that at least $N/2$ children are good.
For each of them, Theorem~\ref{thm:ELO} bounds the conditional
probability of not being $(\lambda/n)$-heavy by $Cn^{-1/20}$.
Thus the expected number of unsuccessful good children is at most
$CNn^{-1/20}$. Fewer than $N/10$ heavy children would require more than
$2N/5$ unsuccessful good children, so Markov's inequality bounds this
probability by $Cn^{-1/20}$.

Otherwise at least $N/2$ children are bad. Let $I$ index them, and for
$h\in[n]$ define
\begin{equation}\label{eq:degree-h}
d(h):=\#\left\{i\in I:h\in B_i,\quad
\abs{X_{n-j}(B_i\setminus\{h\})}\geq\lambda/n\right\}.
\end{equation}
Counting the heavy parents of bad children gives
$\sum_{i\in I}d(h_i)<n^{1/10}\abs I$.
Hence, for $I':=\{i\in I:d(h_i)\leq2n^{1/10}\},$
\begin{equation}\label{eq:Iprime-size}
\abs{I'}\geq\abs I/2\geq N/4.
\end{equation}

Condition further on the new row signs outside $\{h_i:i\in I'\}$.
In this conditional law, probability and expectation refer to the
remaining independent signs. Set
\begin{equation}\label{eq:Yi-endgame}
Y:=\sum_{i\in I'}\min\left\{
\frac{\abs{X_{n-j+1}(B_i)}}{\lambda},1\right\}.
\end{equation}
For each $i\in I'$, fix all remaining signs except
$\varepsilon_{h_i}$. In \eqref{eq:Bi-cofactor}, its coefficient
is $X_{n-j}(A_i)$, whose absolute value is at least $\lambda$.
Hence at least one sign choice makes the $i$th summand in
\eqref{eq:Yi-endgame} equal to one. Consequently,
\begin{equation}\label{eq:EY-lower}
\E Y\geq\abs{I'}/2.
\end{equation}

Flipping a remaining sign $\varepsilon_h$ changes at most $d(h)$
summands by at most one each, and every other summand by at most $2/n$.
Since $\abs{I'}\leq n$, the total change is at most
\begin{equation}\label{eq:Y-Lipschitz}
2n^{1/10}+2\abs{I'}/n\leq3n^{1/10}
\end{equation}
for large $n$.
Theorem~\ref{thm:mcdiarmid}, with $t=\abs{I'}/20$, therefore gives
\begin{equation}\label{eq:Y-concentration}
\Prob\left(Y<\frac{9\abs{I'}}{20}\right)
\leq\exp\left(-\frac{\abs{I'}}{1800n^{1/5}}\right)
\leq\exp(-cNn^{-1/5}).
\end{equation}
A child with $\abs{X_{n-j+1}(B_i)}<\lambda/n$
contributes less than $1/n$ to the sum in \eqref{eq:Yi-endgame}.
On $\{Y\geq9|I'|/20\}$, the number of $(\lambda/n)$-heavy
children among $\{B_i:i\in I'\}$ is therefore at least
\[
Y-\frac{\abs{I'}}n
\geq\left(\frac9{20}-\frac1n\right)\abs{I'}
\geq\frac N{10}
\]
for large $n$. Averaging \eqref{eq:Y-concentration} over the additionally fixed
signs shows that, conditionally on $\cF_{n-j}$, the probability
of obtaining fewer than $\floor{N/10}$ $(\lambda/n)$-heavy
children is at most $\exp(-cNn^{-1/5})$.

In both cases the retained children have disjoint complements, and the
conditional failure probability is at most
\begin{equation}\label{eq:codim-failure-raw}
Cn^{-1/20}+\exp(-cNn^{-1/5})\leq C_0n^{-c_0},
\end{equation}
since $N\geq n^{1/2}$.
\end{proof}

\subsection{Completing the endgame}

\begin{proof}[Proof of Proposition~\ref{prop:endgame}]
Condition on $\cF_{k_*}$ on the event that $A_*$ is $\lambda$-heavy.
Since $r_*\geq3L$ for large $n$, Proposition~\ref{prop:many-codim-L}
applies. For $1\leq j\leq L$, set
\begin{equation}\label{eq:Nj-floor-bound}
N_j:=\floor{\frac{r_*}{4L\,10^{L-j}}}
\geq\frac{r_*}{4L\,10^{L-j}}-1.
\end{equation}
Then $N_{j-1}=\floor{N_j/10}$, and
$\mathcal E(L,N_L,\lambda)$ holds except with conditional probability
$Ce^{-cL}$. Since $10^L\leq n^{(\log10)/100}$,
\begin{equation}\label{eq:N1-large}
N_j\geq N_1
\geq c\frac{n^{1-(\log10)/100}}{(\log n)^2}-1
\geq n^{0.9}
\end{equation}
for large $n$, so Lemma~\ref{lem:codimension-reduction} applies at every
step. Iterating for $j=L,L-1,\ldots,2$ divides the threshold by $n$ at
each step. A first-failure union bound therefore gives
$\mathcal E(1,N_1,\lambda n^{-(L-1)})$ except with conditional
probability at most $Ce^{-cL}+C_0(L-1)n^{-c_0}$.

On $\mathcal E(1,N_1,\lambda n^{-(L-1)})$, the final row expansion
\begin{equation}\label{eq:final-row-expansion}
\Per(M_n)=\sum_{i=1}^n\xi_{n,i}X_{n-1}([n]\setminus\{i\})
\end{equation}
has at least $N_1$ coefficients of absolute value at least
$\lambda n^{-(L-1)}$. Conditioning on $\cF_{n-1}$ and applying
Theorem~\ref{thm:ELO} gives
\begin{equation}\label{eq:final-ELO}
\Prob\left(
\abs{\Per(M_n)}<\lambda n^{-(L-1)}
\,\middle|\,\cF_{n-1}
\right)\leq CN_1^{-1/2}\leq Cn^{-0.45}.
\end{equation}
Averaging \eqref{eq:final-ELO} on
$\mathcal E(1,N_1,\lambda n^{-(L-1)})$
and adding the failure probability of the iteration gives
the following bound conditional on $\cF_{k_*}$:
\begin{equation}\label{eq:endgame-total-failure}
Ce^{-c_1L}+C_0(L-1)n^{-c_0}+Cn^{-0.45}\leq n^{-c}
\end{equation}
for some absolute $c>0$, because $L=\floor{\log n/100}$. The proposition is proved since $\lambda n^{-(L-1)}\geq\lambda n^{-L}$.
\end{proof}

\section{Proof of the main theorem}
\label{sec:main-proof}

Propositions~\ref{prop:large-near-top-minor} and~\ref{prop:endgame}
yield an $e^{-Cn}\sqrt{n!}$ lower bound for the full permanent
with polynomially small failure probability.

\begin{proof}[Proof of Theorem~\ref{thm:main}]

By Proposition~\ref{prop:large-near-top-minor}, there are absolute
constants $C_1,c_1>0$ and an $\cF_{k_*}$-measurable choice
$A_*\in\binom{[n]}{k_*}$ such that
\begin{equation}\label{eq:near-top-event-prob}
\Prob\left(\abs{X_{k_*}(A_*)}<e^{-C_1n}\sqrt{k_*!}\right)
\leq e^{-c_1n/(\log n)^2}.
\end{equation}
On $\{\abs{X_{k_*}(A_*)}\geq e^{-C_1n}\sqrt{k_*!}\}$,
apply \eqref{eq:endgame-conclusion} with
$\lambda=e^{-C_1n}\sqrt{k_*!}$ to obtain
\begin{equation}\label{eq:after-endgame}
\Prob\left(
\abs{\Per(M_n)}<n^{-L}e^{-C_1n}\sqrt{k_*!}
\,\middle|\,\cF_{k_*}
\right)\leq n^{-c_2}
\end{equation}
for an absolute $c_2>0$.
Since $n!/k_*!\leq n^{r_*}$,
\begin{equation}\label{eq:kstar-factorial-comparison}
n^{-L}e^{-C_1n}\sqrt{k_*!}
\geq\exp\left\{-C_1n-\left(\frac{r_*}{2}+L\right)\log n\right\}
\sqrt{n!}.
\end{equation}
Since $r_*\leq n/\log n+1$ and $L\leq\log n/100$,
\begin{equation}\label{eq:completion-scale-loss}
\left(\frac{r_*}{2}+L\right)\log n
\leq\frac n2+\frac{\log n}{2}+\frac{(\log n)^2}{100}=O(n).
\end{equation}

By \eqref{eq:kstar-factorial-comparison} and
\eqref{eq:completion-scale-loss},
$n^{-L}e^{-C_1n}\sqrt{k_*!}\geq e^{-Cn}\sqrt{n!}$
for an absolute $C>0$.
Averaging \eqref{eq:after-endgame} on
$\{\abs{X_{k_*}(A_*)}\geq e^{-C_1n}\sqrt{k_*!}\}$
and using \eqref{eq:near-top-event-prob}, we obtain
\[
\Prob\left(\abs{\Per(M_n)}<e^{-Cn}\sqrt{n!}\right)
\leq e^{-c_1n/(\log n)^2}+n^{-c_2}\leq n^{-c}
\]
for a suitable absolute $c>0$ and all sufficiently large $n$.
\end{proof}

\subsection*{Acknowledgments}  This work is supported by the National Natural Science Foundation of China
(12595294, 12231002), and the New Cornerstone Science Foundation (NCI202501).

%

\end{document}